\documentclass[12pt,a4paper,leqno]{amsart}

\usepackage{amssymb,hyperref}
\usepackage{array,float}
\usepackage{amsmath,amsthm,amsbsy,amscd,mathrsfs}
\usepackage{graphicx}
\usepackage{caption}
\usepackage{cite}
\usepackage{multicol}

\usepackage{mathtools}

\DeclareMathOperator{\Pic}{Pic}

\DeclareMathOperator{\co}{covol}

\makeatletter
\newcommand\suchthat{%
 \@ifstar
  {\mathrel{}\middle|\mathrel{}}
  {\mid}%
}
\makeatother

\theoremstyle{plain}
 \newtheorem{theorem}{Theorem}[section]
 \newtheorem{proposition}{Proposition}[section]
 \newtheorem{lemma}{Lemma}[section]
 \newtheorem{corollary}{Corollary}[section]
\theoremstyle{definition}

\theoremstyle{remark}
\newtheorem{remark}{Remark}[section] 
\numberwithin{equation}{section}

 \newtheorem{definition}{Definition}[section]

\title[REDUCED IDEALS FROM THE REDUCTION ALGORITHM]{REDUCED IDEALS FROM THE REDUCTION ALGORITHM}

\author[Ha Thanh Nguyen Tran]{Ha Thanh Nguyen Tran} 
\address{Department of Mathematics and Statistics,
University of Calgary, Canada.}
\email{hatran1104@gmail.com}

\author[Duong Hoang Dung]{Duong Hoang Dung} 
\address{ School of Computing and Information Technology,
	University of Wollongong
	NSW, Australia,  2522.}
\email{hduong@uow.edu.au}



\begin{document}

\maketitle

\begin{abstract}\textbf{}
The reduction algorithm is used to compute reduced ideals of a number field. However, there are reduced ideals that can never be obtained from this algorithm. In this paper, we will show that these ideals have inverses of larger norms among reduced ones. Especially, 	we represent a sufficient and necessary condition for reduced ideals of real quadratic fields to be obtained from the reduction algorithm. 
\end{abstract}


\section{Introduction}
Reduced ideals of a number field $F$ have inverses of small norms and they form a finite and regularly distributed set in the infrastructure of $F$. Therefore, they can be used to compute the regulator and the class number of a number field \cite{ref:7,ref:5,ref:8,ref:10,ref:9,ref:4}. One usually applies the reduction algorithm (see Algorithm 10.3 in \cite{ref:4}) to find them. Ideals obtained from this algorithm are called 1-reduced \cite{Tran1}. 
There exist reduced ideals that are not 1-reduced. 
For example, the real quadratic field $F= \mathbb{Q}(\sqrt{73})$ has nine reduced ideals but only seven of them are 1-reduced. The ideals $D_2$ and $D_9$ are reduced but not 1-reduced. 

\begin{figure}[!htb]
	\centering
	\begin{minipage}{.5\textwidth}
		\centering
		\includegraphics[width=0.6\linewidth]{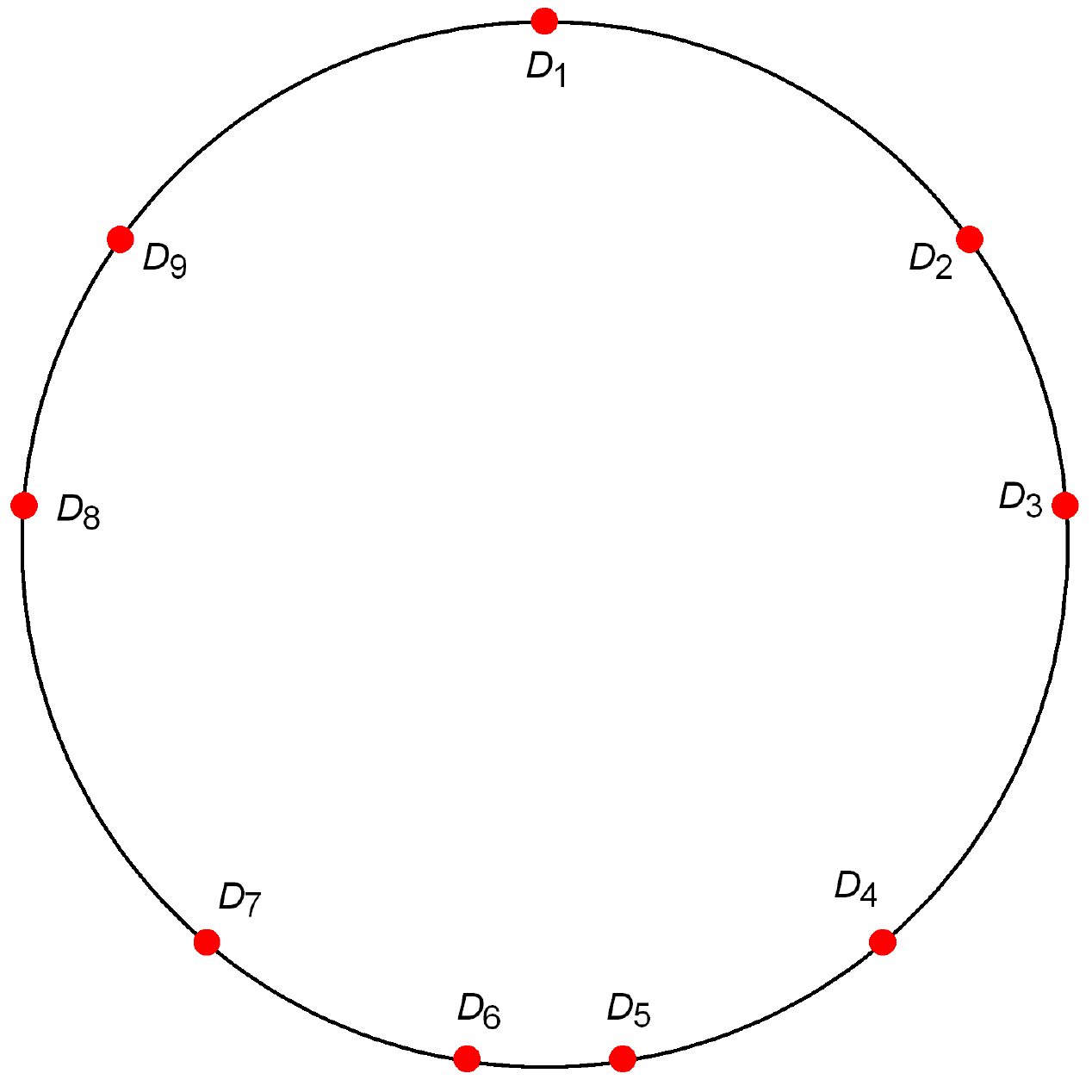}
		\caption{The reduced ideals of $F= \mathbb{Q}(\sqrt{73})$. \label{pic:red_disc73}}
	\end{minipage}%
	\begin{minipage}{0.5\textwidth}
		\centering
		\includegraphics[width=0.6\linewidth]{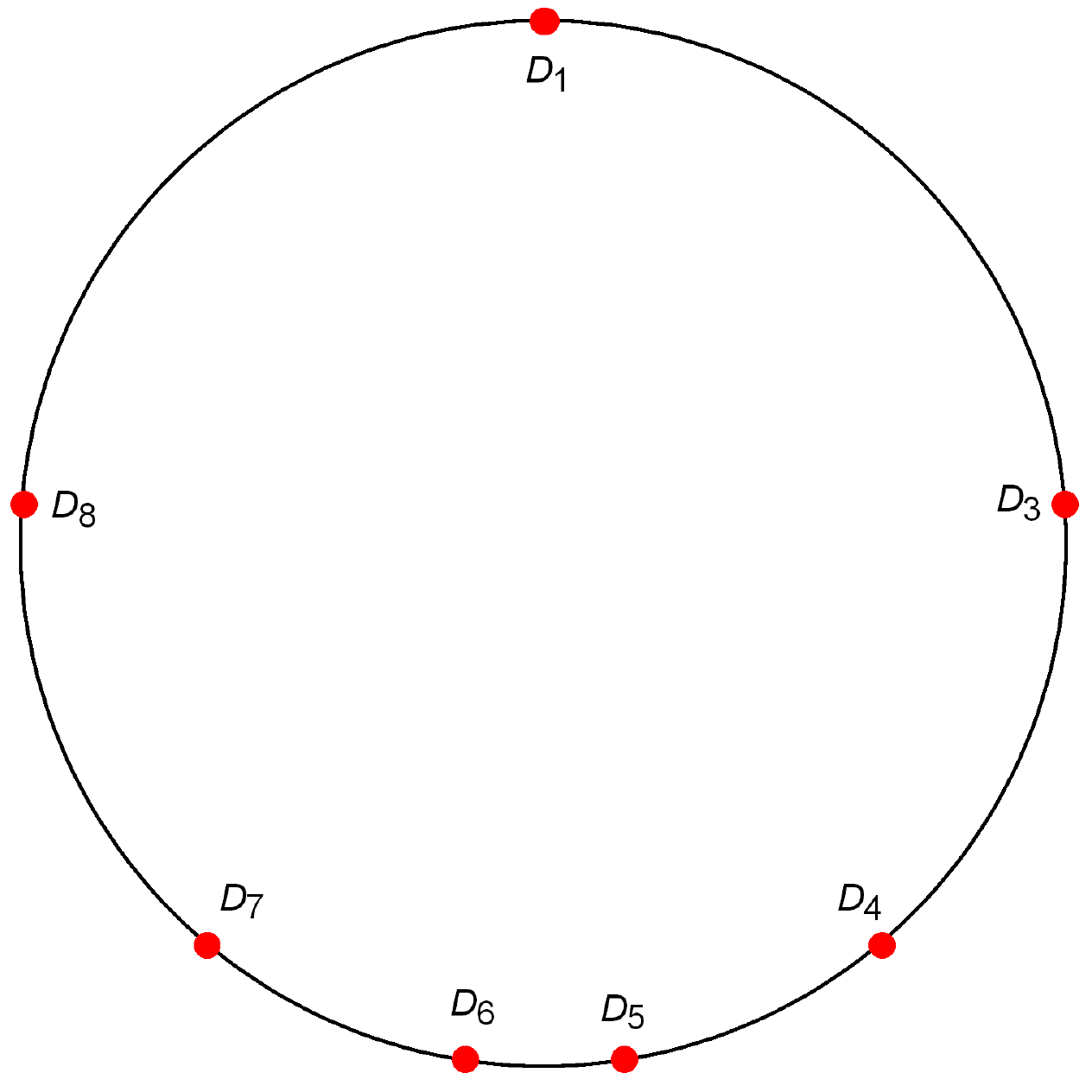}
		\caption{The 1-reduced ideals of $F= \mathbb{Q}(\sqrt{73})$. \label{pic:1red_disc73}}
	\end{minipage}
\end{figure}

In this paper, we first show that for an arbitrary number field, the inverses of $1$-reduced ideals must have small norms compared to the discriminant of $F$. The result is represented in Section \ref{sec4}.

In case of real quadratic fields, we prove a sufficient and necessary condition for a reduced ideal to be 1-reduced. Explicitly, each reduced ideal contains a unique element $f$ satisfying the conditions given in Remark \ref{remrealquad}. We will prove the following theorem.
\begin{theorem}\label{thmrealquad}
	Let $F$ be a real quadratic field and let $I= \mathbb{Z} + \mathbb{Z} f$ be a reduced fractional ideal of $F$ where $f$ is in Remark \ref{remrealquad}.
	Then $I$ is 1-reduced if $N(f-1/2) \le -3/4$.
	
\end{theorem}
 The upper bound in this theorem is actually tight. In other words, there exist 1-reduced ideals such that $N(f-1/2) = -3/4$. For instance, this equality is obtained in the quadratic field
	$F=\mathbb{Q}(\sqrt{511})$ for the 1-reduced ideal $I=1\cdot \mathbb{Z} + \frac{19+\sqrt{511}}{25}\cdot \mathbb{Z}$, and in the quadratic field 
	$F=\mathbb{Q}(\sqrt{3})$ for the 1-reduced ideal $I=1\cdot \mathbb{Z} + \frac{1+\sqrt{3}}{2}\cdot \mathbb{Z}$. Especially, the later one is the only case of which the inverse ideal has smallest norm, that is  $N(I^{-1})=\sqrt{|\Delta_F|/3}$ where $\Delta_F$ is the discriminant of $F$.

In addition, our results partly answer the question mentioned in Section 12 in \cite{ref:4} which asks about the number of reduced but not 1-reduced ideals. Indeed, the results of Corollary \ref{corolquad} and Proposition \ref{prophigher} imply that reduced but not 1-reduced ideals are the ones have inverses of large norms among reduced ideals. We further more, are interested in finding properties of reduced but not 1-reduced ideals, estimating how many of them as well their distribution in the topological group $\Pic^0_F$.

Note that we do not consider 1-reduced ideals of imaginary quadratic fields. That is because in these fields, reduced ideals are always 1-reduced.

\section{Preliminaries}\label{sec2}
In this section, let $F$ a number field of degree $n$ and the discriminant $\Delta_F$. Assume that $F$ has $r_1$ real embeddings $\sigma_1, \cdots, \sigma_{r_1}$ and $r_2$ complex embeddings (up to conjugation) $\sigma_{r_1 +1}, \cdots, \sigma_{r_1+r_2}$. Thus $n= r_1 + 2 r_2$. Denote by $\Phi = (\sigma_1, \cdots, \sigma_{r_1},\sigma_{r_1 +1}, \cdots, \sigma_{r_1+r_2})$. Each fractional ideal $I$ of $F$ can be viewed as the lattice $\Phi(I)$ in $F_{\mathbb{R}}=\mathbb{R}^{r_1} \times \mathbb{C}^{r_2}$. 
Now let $u= (u_i) \in \mathbb{R}^{r_1+r_2}$. Then we define $N(u)= \prod_i u_i$. We also identify each element $f \in I$ with $\Phi(f) = (\sigma_i(f))_i \in F_{\mathbb{R}}$ and use the  following metric in $F_{\mathbb{R}}$.
$$\|u f\|^2 = \sum_{i=1}^{r_1} u_i^2 |\sigma_i(f)|^2 + 2 \sum_{i=r_1+1}^{r_1+r_2} u_i^2 |\sigma_i(f)|^2.$$


\subsection{Reduced ideals}
\begin{definition}
	A fractional ideal $I$ is called \textit{reduced} if $1 $ is minimal in $I$. In other words, if $g \in I$ and $|\sigma_i(g)| < 1$ for all $i$ then $g=0$.
\end{definition}

\begin{definition}
	Let $I$ be a fractional ideal. Then 1 is called  
	\textit{primitive } in $I$ if $1 \in I$ and it is not divisible by any integer $d \geq 2$.
\end{definition}

\begin{definition}
	Let $I $ be a fractional ideal in $F$ and let $u\in (\mathbb{R}_{>0})^{r_1+r_2}$. 
	The \textit{length of an element $g$ of $I$ with respect  to $u$} is defined by $\|g\|_u:= \|u g\| $.
\end{definition}

\begin{definition}\label{def:1}
	A fractional ideal $I$ is called $1$-\textit{reduced} if:
	\begin{itemize}
		\item $1$ is primitive in $I$, and
		\item  there exists $u \in \prod_{\sigma}\mathbb{R}_{>0}$  such that$  \|1\|_u \leq \|g\|_u $  for all $ g \in I\backslash \{0\}$.
	\end{itemize}
\end{definition}

\begin{remark}\label{aboutu}\qquad
\begin{itemize}
\item The second condition of Definition \ref{def:1} is equivalent to saying that there exists a metric $u$ such that with respect to this metric, the vector $1$ is a shortest vector in the lattice $I$.
\item Since the lattice $L=uI := \{(u_i \sigma_i(x))_i: x \in I \} \subset F_{\mathbb{R}}$  is isometric to the lattice $\Phi(I)$ with respect to the length function $\|\hspace{0.2cm}\|_u$, the second condition of Definition \ref{def:1} is equivalent to saying that $u$ is a shortest vector of the lattice $L$.
\item If $u= (u_i) \in (\mathbb{R}_{>0})^{r_1+r_2}$ satisfies the second condition of Definition \ref{def:1}, then so is $u' = \left(\frac{u_{\sigma}}{N(u)^{1/n}}\right)_{\sigma}\in (\mathbb{R}_{>0})^{r_1+r_2}$ and  $N(u')=1$. Therefore, we can always assume that $N(u) =1$. 
\item If $u \in \mathbb{R}^{r_1+r_2}$ and $N(u)=1$ then $\|u\|^2 \ge n N(u)^{2/n}= n$ by the arithmetic--geometric
mean inequality.
\end{itemize}

\end{remark}

\subsection{The reduction algorithm}\label{redalg}
Given an ideal lattice $I$ with a metric $u$ such that the covolume of this ideal is $\sqrt{|\Delta_F|}$. 
Compute  a reduced ideal $J$ such that $(J, N(J)^{-1/n})$ is close to $(I,u)$ in $\Pic_F^0$ (see Algorithm 10.3 in \cite{ref:4}).

\textbf{Description}. We compute an LLL-reduced basis $b_1, \cdots , b_n$ of the lattice $L= u I \subset F_{\mathbb{R}}$. Then we compute a shortest vector $x$ in $L$ as follows. Any shortest vector $x = \sum_{i=1}^{i=n} m_i b_i$ satisfies $\|x\|/\|b_1\| \le 1$. Therefore the coordinates $m_i \in \mathbb{Z}$ are bounded independent of the discriminant of $F$. To compute a shortest vector in the lattice in time polynomial in $\log{|\Delta_F|}$, we may therefore just try all possible $m_i$. To find a reduced ideal $J$ such that $(J, N(J)^{-1/n})$ is close to $(I,u)$ in $\Pic_F^0$, we compute a shortest vector $f$ in the lattice $(I, u)$. The fractional ideal $J= f^{-1} I$ is then reduced.  In addition, the distance between $(I, u)$ and $(J, N(J)^{-1/n})$ in $\Pic_F^0$ is at most $\log{|\Delta_F|}$.

\begin{remark}
The ideal $J$ obtained from the reduction algorithm above is 1-reduced. First, it is easy to show that $1$ is primitive in $J$. Now let $v = u |f|: = (u_i\cdot |\sigma_i(f)|)_i  \in (\mathbb{R}_{>0})^{r_1+r_2}$. We then have the following.
$$\|1\|_v = \| v \| = \|u |f|\| = \|u f\| =\|f\|_u.$$
Any element $h$ of $J$ has the form $h= f^{-1} g$ for some $g \in I$. Thus
$$\|h\|_v =\| f^{-1} g\|_v= \|f^{-1} g v\|=\|f^{-1} g u |f| \|_v = \|g u\|= \|g\|_u.$$
Since $f$ is a shortest vector in the lattice $I$ with respect to the metric $u$, we have $\|f\|_u \le \|g\|_u \text{ for all } g\in I\backslash \{0\}$. Therefore 
$$\|1\|_v  \le \|h\|_v  \text{ for all } h \in J\backslash \{0\}.$$
Thus $J$ is 1-reduced.
\end{remark}

\section{A result for an arbitrary field}\label{sec4}
	
Using Remark \ref{aboutu}, we can prove the following result where $\gamma_n$ is the Hermite constant in dimension $n$ \cite{blichfeldt}.

\begin{proposition}\label{prophigher}
	Let $c_n= \left(\frac{n}{\gamma_n}\right)^n$ and let $I$ be a fractional ideal containing $1$. Then $I$ is not 1-reduced if $N(I^{-1}) > \sqrt{|\Delta_F|/c_n}$.
\end{proposition}

\begin{proof}
	Suppose that $I$ is $1$-reduced. Then there is some $u \in (\mathbb{R}_{>0})^n$ such that $u$ is a shortest vector in the lattice $L= uI$. Thus its length is bounded as follows.
	$$\| u\|^2 \le \gamma_n \co(L)^{2/n}.$$
	Since $N(I^{-1}) > \sqrt{|\Delta_F|/c_n}$, we have $N(I) < \sqrt{c_n/|\Delta_F|}$. We can assume that $N(u)=1$ as in the Remark \ref{aboutu}. It follows that
	$\co(L) = N(u) N(I) \sqrt{|\Delta_F|}= N(I) \sqrt{|\Delta_F|}< \sqrt{c_n}$. Therefore
	$\|u \|^2 < \gamma_n c_n^{1/n} = n$, contradicting the fact that $\|u\|^2 \ge n$ as in Remark \ref{aboutu}. 
	Hence $I$ is not $1$-reduced.
	
\end{proof}

The table below shows values of $c_n$ corresponding to known values of $\gamma_n$ (in dimensions 1 to 8 and 24).
\vspace*{0.2cm}
        \label{table1}
				\begin{center}
				\resizebox{0.5\textwidth}{!}{
          \begin{tabular}{|c| c |c |c| c| c| c |c| c|} 
               	\hline 
               	$n$ & 2 & 3 & 4 & 5 & 6 & 7 & 8 & 24 \\ 
													
               	\hline                	
               	$c_n$ & 3 & $27/2$ & $2^6$ & $5^5/8$ & $3^7$ & $7^7/64$ & $4^8$ & $6^{24}$\\
               	\hline
        \end{tabular}
				}
				\end{center}
\vspace*{0.2cm}
Note that our result in Proposition \ref{prophigher} agrees with Theorem \ref{thmrealquad} in case $n=2$, i.e., the norm of the inverse of a $1$-reduced ideal must be less than or equal to $\sqrt{|\Delta_F|/3}$.

\section{Real quadratic fields}\label{sec3}
In this section, let $F$ be a real quadratic fields with two real embeddings $\sigma$ and $\sigma'$ that send $\sqrt{\Delta_F}$ to $\sqrt{\Delta_F}$ and $-\sqrt{\Delta_F}$ respectively. We denote by $\Phi = (\sigma, \sigma')$ the map from $F$ to $\mathbb{R}^2$.

\subsection{Reduced ideals of real quadratic fields}

Let $I$ be a fractional ideal of $F$ and let $u= (u_1, u_2) \in \mathbb{R}^2$. We identify each element $f \in I$ with $\Phi(f) = (\sigma(f), \sigma'(f)) \in \mathbb{R}^2$ and use the standard metric in $\mathbb{R}^2$ as follows. 
$$\|u g \|^2 = u_1^2 [\sigma(f)]^2 + u_2^2 [\sigma'(f)]^2.$$

\begin{remark}\label{remrealquad}	\qquad
	\begin{itemize}
		\item Any reduced ideal $I$ of $F$ can be written as the following form
		$$I = \mathbb{Z} + f \mathbb{Z} \text{ for a unique   } f \in F \text{ satisfying } \sigma(f)>1 \text{  and } -1<\sigma'(f)<0.$$
		In particular, $f$ can be written as 
		\begin{equation}\label{eqabc}
		f =\frac{b+\sqrt{\Delta_F}}{2 a},\qquad (a, b, c) \in \mathbb{Z}^3, \qquad \Delta_F=b^2-4 a c \text{ and  } |\sqrt{\Delta_F}- 2a|<b<  \sqrt{\Delta_F}.
		\end{equation}
			Moreover, the inverse of $I$ is an integral ideal, that is $I^{-1} \subset O_F$, and its norm $N(I^{-1}) =a$.
		(See Example 8.2 in \cite{ref:4} for more details.)
		
		Here we view  $I$ as the lattice $\Phi(I)$ in $\mathbb{R}^2$ as below.
		\[
		I \equiv \Phi(I) = 
		\left[ {\begin{array}{cc}
			1 & \sigma(f)\\ 
			1 & \sigma'(f)\\
			\end{array} } \right] \mathbb{Z}^2 =  \left[ {\begin{array}{cc}
			1 & \frac{b+\sqrt{\Delta_F}}{2a}\\ 
			1 & \frac{b-\sqrt{\Delta_F}}{2a}\\
			\end{array} } \right] \mathbb{Z}^2 
		\]
		
		In other words, $I$ is identified to the free $\mathbb{Z}$-module generated by two vectors 
		$$\Phi(1)=(\sigma(1), \sigma'(1)) = (1, 1)	\text{  and  } \Phi(f)=(\sigma(f), \sigma'(f)) =\left(\frac{b+\sqrt{\Delta_F}}{2a}, \frac{b-\sqrt{\Delta_F}}{2a}\right).$$

	\end{itemize}
	
\end{remark}

\subsection{Test 1-reduced property}\label{test}
Assume that $I$ is reduced and the shortest vectors of $I$ have length strictly less than $\sqrt{2}$. In this part, we show a method to test whether $I$ is $1$-reduced or not (see \cite{Tran1} for more details).

Let $g \in I$. We denote by $\Phi(g) = (g_1, g_2) \in \mathbb{R}^2$ where $g_1 =\sigma(g)$ and $g_2 = \sigma'(g)$. Denote by
$$G = \left\{g \in I: \left(g_1^2-1 \right) \left(g_2^2-1 \right) < 0 \text{ and } \|g\| < \frac{8}{ \pi} \right\} = G_1 \cup G_2  \text{ where }$$
$$G_1 = \left\{g \in G: g_1^2-1 <0   \} \text{ and }  G_2 = \{g \in G: g_2^2-1<0   \right \}.$$
 
 For each $g \in G$, we define 
 $$ B(g):=\left( -\frac{ g_1^2-1}{ g_2^2-1}\right)^{1/4}.$$
 Then denote
 \begin{equation}
 B_{min} =
  \begin{cases} 
    \frac{1}{2} \hspace*{3cm} \text{  if  } & G_1 = \emptyset \\ 
    \max \left\{ B(g): g \in G_1 \right\}  \text{   if } & G_1 \neq \emptyset.
    \end{cases}
 \end{equation} 
 \begin{equation}
  B_{max} =
  \begin{cases} 
    2  \hspace*{2.8cm} \text{  if  } & G_2 = \emptyset \\ 
    \min \left\{ B(g): g \in G_2 \right\}  \text{   if } & G_2 \neq \emptyset.
    \end{cases}
 \end{equation}
 The ideal $I$ is then $1$-reduced if and only if $B_{max}  \ge B_{min}$ (see Proposition 3.5 in \cite{Tran1}). The Algorithm 4.1 in \cite{Tran1} provides a method to compute $B_{max}$ and $B_{min}$ as follows.
 
 Let $\{b_1=(b_{11}, b_{12}), b_2=(b_{21}, b_{22})\}$ be an LLL-basis for the lattice $I$. We compute the integers $t_1 \leq t_2$ as the following.
 \begin{itemize}
 	\item  If $0<b_{11}<1$ and $1<|b_{12}|<\sqrt{2}$ then $t_1 \leq t_2$ are between $\frac{-1-b_{22}}{b_{12}}$ and  $\frac{1-b_{22}}{b_{12}}$ . 
 	\item If $1<b_{11}<\sqrt{2}$ and $0<|b_{12}|<1$ then $t_1 \leq t_2$ are between $\frac{-1-b_{21}}{b_{11}}$ and $\frac{1-b_{21}}{b_{11}}$.
 \end{itemize} 
  Then $B_{max}$ and $B_{min}$ are among $B(g)$ where 
  $$g \in  G_3 = \{b_1,  t_1 b_1 +b_2,  t_2 b_1 +b_2,  s_1b_1 + b_2 \text{ with } |s_1| \leq 2\}.$$

Especially, if we further assume that $b_2=(b_{21}, b_{22})= (1,1)$, then $t_1 = t_2=0$. Thus, the set $G_3$ can be $\cdots$ as
$$ G_4 = \{b_1, b_1 + b_2,  b_2 -b_1, 2b_1 + b_2,  b_2 -2b_1\}.$$

\subsection{Proof of Theorem \ref{thmrealquad}}

 	Note that the condition $N(f-1/2) \le -3/4$ is equivalent to the following.
 	\begin{equation}\label{eqconditionb}
 	(b-a)^2+ 3a^2 \le \Delta_F.
 	\end{equation}
 	The condition \ref{eqconditionb} implies that $3a^2 \le \Delta_F$. It is also equivalent to the following.
 	$$a-\sqrt{\Delta_F-3a^2} \le b \le a+\sqrt{\Delta_F-3a^2}.$$
 	Therefore, we can divide the proof into three cases as below. 
 	\begin{itemize}
 		\item \textbf{Case 1:}  $a\le \sqrt{\Delta_F/4}$ or\\
 		$\sqrt{\Delta_F/4} \le a \le  \sqrt{\Delta_F/3}$ and $\sqrt{4a^2-\Delta_F} \le b \le 2a-\sqrt{4a^2-\Delta_F}$.\\
 		By Lemma \ref{case12}, $1$ is a shortest vector of the lattice $I$. Therefore it is $1$-reduced.

 		\item \textbf{Case 2:} $\sqrt{\Delta_F/4} \le a \le  \sqrt{\Delta_F/3}$ and $a-\sqrt{\Delta_F-3a^2} \le b \le \sqrt{4a^2-\Delta_F} $. In this case, the two vectors $\Phi(f)$ and $\Phi(1)$ form an LLL-reduced basis for $I$ (see Lemma \ref{case3}). Hence, the result of Section \ref{test} can be used to show that $I$ is 1-reduced. We first compute $B_{max}$ and $B_{min}$ among $B(g)$ where
 		$$ g \in  G_4 = \{\Phi(f), \Phi(f) + \Phi(1), -\Phi(f)+\Phi(1), 2\Phi(f) + \Phi(1),  -2\Phi(f)+ \Phi(1)\}.$$
 		Since the vector $-2\Phi(f)+ \Phi(1)$ has both coordinates greater than 1, we can eliminate it from the set $G_4$. Furthermore, we obtain that \\
 		$$B_{min}= B( -\Phi(f) + \Phi(1))= \frac{(\sqrt{\Delta_F} +b)(4a-b-\sqrt{\Delta_F})}{(\sqrt{\Delta_F}-b)(4a-b+\sqrt{\Delta_F})},$$
 		
 		$$B_{max}= \min\{ B(\Phi(f)), B(\Phi(f) + \Phi(1)), B(2\Phi(f) + \Phi(1))\}.$$
 		
 		$$B(\Phi(f))= \frac{(2a+b+\sqrt{\Delta_F})(b+\sqrt{\Delta_F}-2a)}{(2a+b-\sqrt{\Delta_F})(2a-b+\sqrt{\Delta_F})},$$
 		
 		$$B(\Phi(f) +\Phi(1))= \frac{(\sqrt{\Delta_F} +b)(4a+b+\sqrt{\Delta_F})}{(\sqrt{\Delta_F}-b)(4a+b-\sqrt{\Delta_F})},$$

 		$$B(2\Phi(f) +\Phi(1))= \frac{(\sqrt{\Delta_F} +b)(2a+b+\sqrt{\Delta_F})}{(\sqrt{\Delta_F}-b)(2a+b-\sqrt{\Delta_F})}.$$
 		
 		By the condition $b \le \sqrt{\Delta_F}$, it is obvious that $B_{min} \le B(\Phi(f) +\Phi(1))$. In addition, $ B_{min} \le B(2\Phi(f) +\Phi(1))$ since $4a-b-\sqrt{\Delta_F} < 2a+b+\sqrt{\Delta_F}$ and 
 		$4a-b+\sqrt{\Delta_F} > 2a+b-\sqrt{\Delta_F}$.
 		
 		Using the fact that $b < \sqrt{\Delta_F} \le 2a$  and the condition \ref{eqconditionb}, all the factors of the following difference
 		$$B(\Phi(f))- B_{min} = \frac{8 a \sqrt{\Delta_F} \hspace*{0.2cm}[\Delta_F- 3a^2- (a-b)^2 ]}{(\sqrt{\Delta_F}-b)[4a^2-(\sqrt{\Delta_F}-b)^2] [4a+\sqrt{\Delta_F}-b]}$$
 		are non negative. In other words,  $ B_{min} \le B(\Phi(f))$.
 		Therefore $B_{min}  \le B_{max}$, then Section \ref{test} says that $I$ is 1-reduced.
 		
 		\item \textbf{Case 3:} $\sqrt{\Delta_F/4} \le a \le  \sqrt{\Delta_F/3}$ and $2a-\sqrt{4 a^2-\Delta_F} \le b \le a + \sqrt{\Delta_F- 3a^2} $. Lemma \ref{case4} shows that $\Phi(f-1)$ and $\Phi(1)$ form an LLL-reduced basis for $I$. By using an argument similar to the proof of Case 2, we obtain the following.	
 		$$B_{max}= B(\Phi(f))  \qquad \qquad \text{ and }$$		
 		$$B_{min}= \max\{ B(-\Phi(f) + \Phi(1)), B(-\Phi(f) +2\Phi(1)), B(-2\Phi(f) + 3\Phi(1))\} \qquad \text{ where }$$
 		$$B(-\Phi(f) +2\Phi(1))= \frac{(\sqrt{\Delta_F} +b-2a)(6a-b-\sqrt{\Delta_F})}{(\sqrt{\Delta_F}+2a-b)(6a +\sqrt{\Delta_F}-b)},$$		 
 		$$B(-2\Phi(f) + 3\Phi(1))= \frac{(b+\sqrt{\Delta_F} -2a)(4a-b-\sqrt{\Delta_F})}{(\sqrt{\Delta_F}+2a-b)(4a +\sqrt{\Delta_F}-b)}.$$
 		Similar to Case 2, we have $B_{max}$ is greater or equal to $B(-\Phi(f) + \Phi(1)), B(-\Phi(f) +2\Phi(1)) $ and $B(-2\Phi(f) + 3\Phi(1))$ by the bounds on $\Delta_F$ and $b$. Thus, $B_{max} \ge B_{min}$. Therefore $I$ is 1-reduced by the result of Section \ref{test}.
 		
 	\end{itemize}

To complete our proof, we prove the following results.

 \begin{lemma}\label{case12}
 	With the assumption in Theorem \ref{thmrealquad}, if one of the following holds
 	\begin{enumerate}
 		\item $a\le \sqrt{\Delta_F/4}$, or
 		\item $\sqrt{\Delta_F/4} \le a \le  \sqrt{\Delta_F/3}$ and $\sqrt{4a^2-\Delta_F} \le b \le 2a-\sqrt{4a^2-\Delta_F}$,
 	\end{enumerate}
 	then $1$ is shortest in the lattice $I$.
 \end{lemma}
\begin{proof}
	By Remark \ref{remrealquad}, we can write $I$ as a lattice in $\mathbb{R}^2$ as below.
	\[
	I= \left[ {\begin{array}{cc}
		1 & \frac{b+\sqrt{\Delta_F}}{2a}\\ 
		1 & \frac{b-\sqrt{\Delta_F}}{2a}\\
		\end{array} } \right] \mathbb{Z}^2.
	\]
	The integers $a, b$ satisfy the condition \ref{eqabc} in Remark \ref{remrealquad}.
	
	Let $g \in I $. Then 
	$$g=\left(m + n\cdot \frac{b+\sqrt{\Delta_F}}{2a},  m + n\cdot \frac{b-\sqrt{\Delta_F}}{2a}\right) \text{  for some } (m, n) \in \mathbb{Z}^2.$$ 
	Thus,
	\begin{equation}\label{eqleng}
	\|g\|^2= 2\left[ \left(m+\frac{nb}{2a}\right)^2 + \frac{n^2 \Delta_F}{4 a^2}\right].	 
	\end{equation} 
	\textbf{Case 1:} $a\le \sqrt{\Delta_F/4}$. We will show that $1$ is shortest in the lattice $I$. Equivalently, we will prove that if 
	$\|g\|^2 < 2$ for some $g \in I$ then $g=0$.
	Indeed, if  $\|g\|^2<2$ holds then by $\ref{eqleng}$, we have  
	$$\frac{n^2 \Delta_F}{4 a^2}<1.$$
	Hence $n= 0$ since $\Delta_F \ge 4 a^2$. Moreover, $\|g\|^2=2 m^2 \ge 2$ for all $m \ne 0$. Thus $m=0$ therefore $g=0$. \\
	\textbf{Case 2:} $\sqrt{\Delta_F/4} \le a \le  \sqrt{\Delta_F/3}$ and $\sqrt{4a^2-\Delta_F} \le b \le 2a-\sqrt{4a^2-\Delta_F}$. Similar to Case 1, we also show that $1$ is shortest in $I$. 
	Let $g \in I$ such that $\|g\|^2<2$. Then $n^2 \le 1$ since $\Delta_F \ge 3 a^2$ and by $\ref{eqleng}$. If $n^2=1$, then $\frac{n^2 \Delta_F}{4 a^2} \ge \frac{3}{4}$. Thus 		
	$$\left|m+\frac{nb}{2a}\right| < \frac{1}{2}.$$
	The bounds on $b$ imply that $ 0 \le b/(2a) \le 1/2$. Therefore $m=0$ and then
	$$\|g\|^2= \frac{2(b^2+\Delta_F)}{4 a^2},$$
	that is at least $2$ by the lower bound on $b$ and $\Delta_F$. Thus $n=0$ and hence $g=0$ as Case 1. 
\end{proof}

\begin{lemma}\label{case3}
	With the assumption in Theorem \ref{thmrealquad}, if $\sqrt{\Delta_F/4} \le a \le  \sqrt{\Delta_F/3}$ and $a-\sqrt{\Delta_F-3a^2} \le b \le \sqrt{4a^2-\Delta_F} $, then $f$ is a shortest vector of the lattice $I$. Moreover, the two vectors $\Phi(f)$ and $\Phi(1)$ form an LLL-reduced basis for $I$.
\end{lemma}
\begin{proof}
	Let $g \in I\backslash \{0\}$. With the notations as in the proof of Lemma \ref{case12}, we show that if $\|g\|^2<2$ then $g=\pm f$. Hence $f$ is shortest in $I$. 
	
	Since $\Delta_F/(4 a^2) \ge 3/4$, using a similar argument as in Case 2 of Lemma \ref{case12} leads to $n^2 \le 1$. If $n=0$ then $\|g\|^2 \ge 2 m^2 \ge 2$ because $m \ne 0$. Thus $n= \pm 1$. Hence
	$\left|m \pm \frac{b}{2a}\right| < \frac{1}{2}$ by \ref{eqleng}, then $\left|m  \right| < \frac{1}{2} +\frac{b}{2a}$.
	The fact that $b \le \sqrt{4a^2- \Delta_F} \le a$ implies that $\frac{b}{2a} \le \frac{1}{2}$. Thus $m=0$ and then $g = \pm f$.
	
	Now let 
	$$\mu = \frac{\langle \Phi(f), \Phi(1) \rangle }{\|\Phi(f)\|^2}=\frac{\sigma(f) + \sigma'(f)}{\|f\|^2}= \frac{2 a b}{b^2 + \Delta_F} .$$ Since $\Delta_F \ge 3 a^2$, we have $\sqrt{4a^2-\Delta_F} \le 2a-\sqrt{4a^2-\Delta_F}$. Consequently, one has $b \le  2a-\sqrt{4a^2-\Delta_F}$. Hence $4ab \le b^2 + \Delta_F$, which implies that $|\mu |\le \frac{1}{2}$. Thus, $\{ \Phi(f), \Phi(1)\} $ is an LLL-reduced basis for $I$.
\end{proof}

\begin{lemma}\label{case4}
	With the assumption in Theorem \ref{thmrealquad}, if $\sqrt{\Delta_F/4} \le a \le  \sqrt{\Delta_F/3}$ and $2a-\sqrt{4 a^2-\Delta_F} \le b \le \sqrt{\Delta_F} $, then $f-1$ is a shortest vector of the lattice $I$. Moreover, the two vectors $\Phi(f-1)$ and $\Phi(1)$ form an LLL-reduced basis for $I$.
\end{lemma}
\begin{proof}
	Let $g \in I\backslash \{0\}$. With the notations as in the proof of Lemma \ref{case12}, we show that if $\|g\|^2<2$ then $g=\pm (f -1)$. Hence $f-1$ is shortest in $I$.
	
	 Since $\Delta_F/(4 a^2) \ge 3/4$, using a similar argument as in Case 2 of Lemma \ref{case12} leads to $n = \pm 1$. Hence
	$\left|m \pm \frac{b}{2a}\right| < \frac{1}{2} \text{ then } \left|m  \right| < \frac{1}{2} +\frac{b}{2a}$.
	The bounds on $b$ and  $\Delta_F$ imply that $\frac{b}{2a} < 1$ and hence $|m| \le 1$. 
	Thus 
	$$g \in \{\pm f, \pm( f +1), \pm( f-1)\}.$$
	Since $\Delta_F \ge 3a^2$, the lower bound $2a-\sqrt{4 a^2-\Delta_F}$ on $b$ is at least $\sqrt{4a^2- \Delta_F}$. Therefore  $b^2 \ge 4 a^2-\Delta_F$, which implies 
	$$\|f\|^2 = 2\left(\frac{b^2 + \Delta_F}{4a^2} \right) \ge 2.$$
	It is easy to see that 
	$\| ( f +1)\|^2 \ge \|f\|^2 \ge 2$ by \ref{eqleng}. 
	Hence we must have $g = \pm (f -1)$.

	Now let $\mu = \frac{\langle \Phi(f-1), \Phi(1) \rangle }{\|\Phi(f-1)\|^2}$. One has
	$$|\mu |= \left|\frac{\sigma(f) + \sigma'(f)-2}{\|f-1\|^2}\right|= \frac{2 a (2a-b)}{(2a-b)^2 + \Delta_F} \le \frac{1}{2}.$$
	The last inequality is obtained since $b^2 \ge 4 a^2-\Delta_F$. Thus, $\{ \Phi(f-1), \Phi(1)\} $ is an LLL-reduced basis for $I$.
\end{proof}



\begin{corollary}\label{corolquad}
	Let $F$ be a real quadratic field and let $I= \mathbb{Z} + \mathbb{Z} f$ be a reduced fractional ideal of $F$ where $f $ is in Remark \ref{remrealquad}.
	Then $I$ is not 1-reduced if and only if $N(f-1/2) > -3/4$.
\end{corollary}

\begin{proof}
	It was shown by Example 9.5 in \cite{ref:4} that if $N(f-1/2) > -3/4$, then $I$ is not 1-reduced. Hence, this result is implied from Theorem \ref{thmrealquad}.
\end{proof}

\begin{corollary}\label{corolquad2}
	Let $F$ be a real quadratic field and let $I$ be a fractional ideal of $F$. If $N(I^{-1}) > \sqrt{\Delta_F/3}$ then $I$ is not 1-reduced.
\end{corollary}
\begin{proof}
	This can be easily seen by the inequality \ref{eqconditionb}.
\end{proof}

\section{Conclusion and Open Problems}
Determining when a reduced ideal is 1-reduced can be solved for quadratic fields since their ideals are explicitly and nicely described (see Remark \ref{remrealquad} and \cite{Tran1}). However, there is no such a description for ideals of higher degree number fields. Hence, this will be a challenge for us to work in the future.

In addition, finding properties, the cardinality and the distribution (in the topological group $\Pic^0_F$) of the set of reduced but not 1-reduced ideals of an arbitrary number field is an open problem for further research.

\section*{Acknowledgement}  
The author is financially supported by the Pacific Institute for the Mathematical Sciences (PIMS).


\end{document}